\documentclass[12pt, reqno]{amsart}
\usepackage{amsmath,amsthm}
\usepackage{amsfonts}
\usepackage{amssymb}
\usepackage{comment}
\usepackage{graphicx}
\usepackage{pstricks,pst-plot,pst-node}
\usepackage{pstricks-add} 
\usepackage{xcolor}
\usepackage[left, pagewise, edtable]{lineno}
\usepackage{enumitem}
\usepackage[margin=1in,footskip=0.25in]{geometry}
\providecommand{\U}[1]{\protect\rule{.1in}{.1in}}
\newtheorem{theorem}{Theorem}

\newtheorem{definition}[theorem]{Definition}

\newtheorem{lemma}[theorem]{Lemma}

\newtheorem{proposition}[theorem]{Proposition}
\newtheorem{remark}[theorem]{Remark}

\DeclareMathOperator{\diag}{diag}
\DeclareMathOperator{\smp}{sim}
\DeclareMathOperator{\shr}{shear}

\begin{document}


\title[Wavelet Coorbit Spaces in Dimension 2]{Classifying Wavelet Coorbit Spaces in Dimension 2}

\author[  Vaishakh. K. J]{Vaishakh Jayaprakash} 
\address{Vaishakh Jayaprakash\\Department of Mathematics, Cochin University of Science and Technology, 682022, Cochin, Kerala, India}
\email{vaishakhkj001@gmail.com}

\author[H. F\"uhr]{Hartmut F\"uhr}

\address{Hartmut F\"uhr \\ Lehrstuhl f\"ur Geometrie und Analysis\\ RWTH Aachen University, D-52056 Aachen,
Germany}
\email{fuehr@mathga.rwth-aachen.de}

\author[N. Asharaf]{Noufal Asharaf}
\address{Noufal Asharaf\\Department of Mathematics, Cochin University of Science and Technology, 682022, Cochin, Kerala, India}
\email{noufal@cusat.ac.in}




\begin{abstract}
Coorbit spaces provide a rigorous framework for the assessment of the approximation theoretic properties of generalized wavelet systems. It is therefore useful to understand when two different wavelet systems give rise to the same scales of coorbit spaces. This paper provides an exhaustive answer to this question for the case of continuous wavelet transforms associated with matrix groups in dimension two. 
\end{abstract}





\maketitle
\section{Introduction}\label{sec1}

The continuous wavelet transform was initially defined to act on functions in $L^2(\mathbb{R})$ \cite{MR747432}. Early on, it was recognized that this transform and its inversion formula were closely related to the representation theory of the \textit{affine group}, the semidirect product $\mathbb{R} \rtimes \mathbb{R}^*$ \cite{MR868528}. This realization opened the door to the extension of the construction to higher dimensions, by considering groups of the type $\mathbb{R}^d \rtimes H$ acting naturally on $L^2(\mathbb{R}^d)$, using translations and dilations from the \textit{dilation group} $H \le GL(d,\mathbb{R})$. The first generalization used the so-called \textit{similitude group} $H = \mathbb{R}^+ \cdot SO(d)$ \cite{MR1010910}, but it was observed early on that this approach can be adapted to a large variety of matrix groups \cite{MR1377491}.


The full scope of this observation became apparent through many different dilation groups and associated wavelet transforms that have since been studied. We refer to \cite{MR1122655, MR2643586, MR1979396, MR2130226,MR1881293} as a small, subjective sample of the literature that has been devoted to the topic. 

This large freedom of choice in the construction of continuous wavelet transforms raises the question of properly understanding the impact of the matrix group on crucial properties of the associated wavelet systems. For example, one may ask which of the different groups give rise to fundamentally different wavelet systems, with the caveat that the phrase ``fundamentally different'' is yet to be defined.

In this paper, we use the approximation-theoretic properties of a wavelet system as a means of comparison.  It was recognized early on that wavelet bases can be used to characterize smoothness spaces such as Besov spaces via the decay properties of the associated wavelet coefficients \cite{MR1990555,MR1085487}. In a very concrete and sharp sense, Besov space norms can characterize signals that can be efficiently approximated by a wavelet system. This fact has been recognized as being extremely useful for the development of mathematical justifications and heuristics for wavelet-based data compression and denoising algorithms. We refer to Remark \ref{rem:Besov_imag} for further elaboration and background.

When replacing the isotropic scaling (or, equivalently, the similitude group) that underlies the wavelet systems characterizing Besov spaces by a more general dilation group $H$, of the kind considered in this paper, the question then arises whether a similar formalism describing the approximation behaviour of the resulting wavelet systems is still available. This is precisely the motivation for the construction of coorbit spaces \cite{bib2,MR1021139}. Coorbit space theory provides a scale of spaces that have the same properties, in relation to the wavelet systems associated to $H$, as the Besov spaces have in relation to wavelet systems associated to the similitude group. More background and further elaborations on these connections may be found in Remarks \ref{rem:weighted_coorbit} and \ref{rem:Besov_coorbit} below.


To summarize, coorbit spaces provide a relevant criterion for the systematic comparison of wavelet systems in higher dimensions: We define two admissible matrix groups $H_1, H_2< {\rm GL}(d,\mathbb{R})$ to be \textit{coorbit equivalent}, if their coorbit spaces coincide, for all coefficient spaces $Y=L^p$, i.e., if $Co_{H_1}(L^p) = Co_{H_2}(L^p)$, $1 \le p \le \infty$. Informally speaking, the associated wavelet systems will then have the same spaces of well-approximated signals. Note that this comparison does not provide any \textit{ranking} in terms of usefulness of the various systems for concrete applications, we just aim to clarify when two wavelet systems can even be expected to yield different performances, based on their approximation theoretic properties.

This paper provides a full classification for the irreducible setting in dimension two. As it turns out, the two-dimensional case and the associated family of candidate groups are still sufficiently tame, so that by proper applications of the results obtained in \cite{MR4669340, MR4832531}, a full classification up to coorbit equivalence can be obtained with moderate technical effort. Our main result is formulated as follows:
\begin{theorem} \label{thm:main}
\begin{enumerate}
\item[(a)] Let $H<GL(2,\mathbb{R})$ denote an irreducibly admissible matrix group with associated open dual orbit $\mathcal{O}$. Then $\mathcal{O}$ has either 1,2, or 4 connected components. 
\item[(b)] Let $H_1, H_2 < GL(2,\mathbb{R})$ denote two irreducibly admissible matrix groups, with associated open dual orbits $\mathcal{O}_1$ and $\mathcal{O}_2$, respectively. Then $H_1 \sim_{Co} H_2$ if and only if $\mathcal{O}_1 = \mathcal{O}_2$, and either 
\begin{enumerate}
    \item[(i)] $\mathcal{O}_1$ has 1 or 4 connected components; or 
    \item[(ii)] $\mathcal{O}_1$ has 2 connected components, and $H_1$ and $H_2$ have the same connected component of the identity. 
\end{enumerate}
\item[(c)] Two distinct irreducibly admissible matrix groups in dimension two are coorbit equivalent only if each contains a subgroup conjugate to the similitude group, or if they share a common open subgroup of finite index. 
\item[(d)] A set of representatives modulo coorbit equivalence is given by the similitude group together with two families of groups that are naturally indexed by suitable pairs of real parameters. 
\end{enumerate}
\end{theorem}
We refer the reader to Section \ref{sect:prelim} for more details on irreducibly admissible matrix groups and their associated coorbit spaces. An explicitly parameterized set of representatives proving part (d) will be given at the end of Section \ref{sect:main}.

Our main result will be proved by suitably combining a previously obtained classification of irreducibly admissible groups modulo \textit{conjugacy} (which is an equivalence relation that is related to but distinct from coorbit equivalence) with more recent criteria for coorbit equivalence, and related notions such as dilational symmetries of coorbit spaces. 
\section{Preliminaries}
\label{sect:prelim}
\subsection{Irreducibly admissible matrix groups}
A closed matrix subgroup $H < GL(d,\mathbb{R})$ is called a \textit{dilation group}.
 Given any dilation group $H$, we let $G:= \mathbb{R}^d \rtimes H$, the semidirect product of $H$ with $\mathbb{R}^d$, generated by arbitrary translations and dilations with elements of $H$. We denote the elements of $G$ by pairs $(x,h) \in \mathbb{R}^d \times H$. For $(x,h), (y,g)\in G$, the group law is given by 
\[
 (x,h)\circ (y,g) = (x+hy,hg)~.
\] We denote the left Haar measure on $G$ by $\mu_G$ and is $d\mu_G(x,h) = dx \frac{dh}{|\mathrm{det}(h)|^2}$, where $d h$ is the left Haar measure on $H$. The space  $L^{2}(G)$ is the set of all square-integrable functions with respect to $\mu_G$. The associated quasi-regular representation of $G$ acts on $L^2(\mathbb{R}^d)$ via
\[
\left[  \pi(x,h) f \right] (y) = |{\rm det}(h)|^{-1/2} f(h^{-1}(y-x))~
\]
for $f\in L^{2}(\mathbb{R}^{d})$. The \textit{continuous wavelet transform } of $f\in L^{2}(\mathbb{R}^{d})$ with respect to a nonzero $\psi \in L^{2}(\mathbb{R}^{d})$ is defined as 
\[
W_\psi f(x,h) = \langle f, \pi(x,h) \psi \rangle~.
\] We call $\pi$ is \textit{square-integrable} if $\pi$ is irreducible, and there exists a nonzero $\psi\in L^{2}(\mathbb{R}^d)$ such that $W_{\psi} \psi \in L^{2}(G)$.  In this case, the linear operator $W_{\psi}: L^{2}(\mathbb{R}^{d}) \rightarrow L^{2}(G)$ is a multiple of an isometry. This isometry property then gives rise to the \textit{weak-sense wavelet inversion formula} 
\[
f = \frac{1}{C_\psi} \int_{G}  W_\psi f(x,h) \pi(x,h) \psi \, d\mu_G(x,h)~
\] where $C_\psi >0$ is a constant depending only on $\psi$. We call the dilation group $H$ is \textit{irreducibly admissible} if $\pi$ is square-integrable. 

This property can be sharply characterized using the \textit{dual action} of $H$, which is the map $H \times \mathbb{R}^{d}\to \mathbb{R}^{d}, (h,\zeta)\longmapsto h^{-T}\zeta$. The dilation group $H$ is irreducibly admissible if and only if the dual action has a single open orbit $\mathcal{O}= H^{-T}{\zeta}_{\mathrm{o}}\subset \mathbb{R}^{d}$ of full measure for some ${\zeta}_{\mathrm{o}}\in \mathbb{R}^{d}$, and additionally the isotropy group (or stabilizer) $H_{{\zeta}_{\mathrm{o}}} =\{h: h^{-T}{\zeta}_{\mathrm{o}} = {\zeta}_{\mathrm{o}}\}\subset H$ is compact \cite{bib4}.

\subsection{Coorbit spaces}

The following subsection contains only a rough outline of coorbit space theory. In particular, we will refrain from describing the full range of possible (quasi-)norms that are typically employed in the discussion of wavelet coorbit spaces, such as weighted $L^{p,q}$-spaces. We refer to the earlier sources, such as \cite{MR3378833}, for more complete definitions. 

The coorbit spaces are defined in terms of the decay behaviour of the continuous wavelet transform. Given an integrability exponent $0< p < \infty$, we define $L^p(G)$ in the usual way as the space of functions whose $p^{th}$ modulus power is integrable against the left Haar measure $\mu_G$, with the usual norm and the associated identification of functions that are equal almost everywhere. For $p=\infty$, the integral is replaced by the essential supremum. 

Now, given $0 \le p < \infty$ and a suitable choice of wavelet $\psi$ (e.g. any nonzero Schwarz function $\psi$ with $\widehat{\psi} \in \mathcal{C}_c^\infty( \mathcal{O})$ \cite{MR3378833}), the coorbit space (quasi-)norm of $f \in L^2(\mathbb{R}^d)$ is defined as
\[
\| f \|_{Co_H(L^p)}: = \| W_\psi f \|_{L^p}~,
\] and the associated coorbit space $Co_H(L^p)$ as (the completion) of the set of all $f \in L^2(\mathbb{R}^d)$ for which this (quasi-)norm is finite.

\begin{remark}
It is customary to study more general classes of coorbit spaces associated to weighted mixed $L^p$-spaces as coefficient spaces, whereas we focus on $L^p$-spaces. This choice was made in the interest of a somewhat shorter exposition. It does not affect the scope of our main result, as Remark \ref{rem:coorbit_equiv} below makes clear.

Remark \ref{rem:weighted_coorbit} below provides more construction details regarding coorbit spaces, and also covers weighted mixed $L^p$-coefficient spaces, thereby allowing to describe the full range of Besov spaces as coorbit spaces.

\end{remark}

\noindent We next formalize the property that two irreducibly admisssible matrix groups have the same coorbit spaces:
\begin{definition} \label{defn:coorbit_equiv}
	Let $H_1, H_2 < GL(d,\mathbb{R})$ denote two irreducibly admissible matrix groups. We call $H_1, H_2$ are  \textit{coorbit equivalent}, denoted by $H_1 {\sim}_{\mathrm{Co}} H_2$, if for all $0<p\leq \infty$ and for all $f\in L^{2}(\mathbb{R}^{d})$ we have 
	\begin{equation*}
		\|f\|_{Co_{H_1}({L}^{p})} \asymp \|f\|_{Co_{H_2}({L}^{p})}.
	\end{equation*}
	Here, the norm equivalence is understood in the generalized sense that one side is infinity if and only if the other side is also.
\end{definition} \noindent
\begin{remark} \label{rem:coorbit_equiv}
Definition \ref{defn:coorbit_equiv} is somewhat differently worded than the original definition of coorbit equivalence in \cite[Definition 2.16]{MR4464548}. The latter requires equivalence for a family of mixed $L^{p,q}$-norms, of which the $L^p$-norms that we employ are a proper subfamily.

However, as Theorem \ref{thm:coorbit_equiv_dual_orbits} below makes clear, as soon as the coorbit space norms associated to $Co_{H_1}(L^p)$ and $Co_{H_2}(L^p)$ are equivalent on $L^2(\mathbb{R}^d)$, for {\em one} choice of $1 \le p < 2$, norm equivalence already follows for the full scale of $L^{p,q}$-spaces associated to $H_1$ and $H_2$, respectively. Hence Definition \ref{defn:coorbit_equiv} is indeed equivalent to \cite[Definition 2.16]{MR4464548}, thereby allowing the use of results from \cite{MR4464548,MR4669340} without further justification. Note also that the norm equivalence even extends to {\em weighted} mixed $L^{p,q}$-spaces, due to item $(f)$ of Theorem \ref{thm:coorbit_equiv_dual_orbits}.
\end{remark}

\begin{remark}
 Note that it is justified to ask whether the equivalence of coorbit space norms on $L^2(\mathbb{R}^d)$, which is the criterion used in Definition \ref{defn:coorbit_equiv}, is actually the proper way of stating that $H_1$ and $H_2$ have the same coorbit spaces.

The first caveat here is that generally speaking, coorbit spaces are  constructed as subspaces of so-called {\em reservoir spaces} that depend on the group, see Remark \ref{rem:weighted_coorbit} below. Hence coorbit spaces associated to different groups can never coincide, unless one uses proper identifications or embeddings, which are often cumbersome to describe. Similar issues occur when one identifies Besov spaces as coorbit spaces, see Remark \ref{rem:Besov_coorbit}.

However, there is a natural identification $L^2(\mathbb{R}^d) = Co_{H_1}(L^2) = Co_{H_2}(L^2)$, which allows to realize that a norm equivalence
\[	\|f\|_{Co_{H_1}({L}^{p})} \asymp \|f\|_{Co_{H_2}({L}^{p})}~, \]
 valid for all $f \in L^2(\mathbb{R}^d)$, gives rise to a natural map
 \[
  Co_{H_1}(L^2) \cap Co_{H_1}(L^p) \to  Co_{H_2}(L^2) \cap Co_{H_2}(L^p)~,
 \]
 that is topological with respect to the $Co_{H_i}(L^p)$-norms.
Furthermore, for $p < \infty$ and $i=1,2$ the subspace
 $Co_{H_i}(L^2) \cap Co_{H_i}(L^p)$ is {\em dense} in $Co_{H_i}(L^p)$. Hence there exists a unique bounded extension, which is a natural topological isomorphism $Co_{H_1}(L^p) \to Co_{H_2}(L^p)$. Again, the reasoning can be extended to weighted mixed $L^p$-spaces.

 A final observation in this context is that for $p,q \le 2$, there exists a natural embedding $Co_H(L^{p,q}) \subset Co_H(L^2) = L^2(\mathbb{R}^d)$, see \cite[Remark 2.11]{MR4464548}. In these settings, one can indeed state that if $H_1$ and $H_2$ are coorbit equivalent, then $Co_{H_1}(L^{p,q}) = Co_{H_2}(L^{p,q})$, {\em as subspaces of $L^2(\mathbb{R}^d)$}.
\end{remark}

We point out that there is a systematic way of deciding whether $H_1 \sim_{Co} H_2$, using the dual actions of $H_1, H_2$ and a certain map $\phi: H_1 \to H_2$ derived from these actions; see \cite{MR4464548} for a detailed exposition. For our analysis of the two-dimensional case, we will be able to rely on a small selection of results. Among these, the following is the necessary criterion for coorbit equivalence. The statement is contained in Theorem 4.17 of  \cite{MR4464548}:
\begin{proposition} \label{prop:dual_orbits_coincide}
    Assume that $H_1,H_2$ are irreducibly admissible matrix groups with open dual orbits $\mathcal{O}_1$ and $\mathcal{O}_2$, respectively. If $H_1 \sim_{Co} H_2$, then $\mathcal{O}_1 = \mathcal{O}_2$.
\end{proposition}

\subsection{Dilational symmetries} In this section, we recall the definitions of various symmetry groups that will significantly simplify our discussion. 
\begin{definition}
    Let $H$ denote an irreducibly admissible matrix group with open dual orbit $\mathcal{O}$. We define the linear symmetry group $\mathcal{S}_{\mathcal{O}}$ of $\mathcal{O}$ by 
    \begin{equation*}
        \mathcal{S}_{\mathcal{O}}:= \{A\in GL(d,\mathbb{R}): A^{T}\mathcal{O}= \mathcal{O}\}.
    \end{equation*}
\end{definition}
\begin{definition}
	Let $H<GL(d,\mathbb{R})$ denote an irreducibly admissible matrix group, and $A\in GL(d,\mathbb{R})$. We call $A$ is coorbit compatible with $H$ if for all $0<p\leq \infty$ and for all $f\in L^{2}(\mathbb{R}^{d})$ we have 
	\begin{equation*}
			\|f\|_{Co_H({L}^p)} \asymp \|f\circ A^{-1}\|_{Co_H({L}^{p})}.
	\end{equation*} We let 
	\begin{equation*}
		\mathcal{S}_H: = \{A\in GL(d,\mathbb{R}):  \mbox{$A$ is coorbit compatible with $H$} \}.
	\end{equation*}
    Furthermore, we let 
    \[
    N_H: = \{A\in GL(d,\mathbb{R)}: AHA^{-1}= H\}~,
    \] the normalizer of $H$ in $GL(d,\mathbb{R})$.
\end{definition} 
\begin{remark}
 Clearly $S_H$ is a subgroup of $GL(d,\mathbb{R})$, and we call it the coorbit symmetry group of $H$. By \cite[Theorem 2.28]{MR4669340}, $A\in GL(d,\mathbb{R})$ is coorbit compatible with $H$ if and only if $AHA^{-1} {\sim}_{\mathrm{Co}} H$, which means we have the alternative characterization $\mathcal{S}_H=\{A\in GL(d,\mathbb{R}): AHA^{-1}{\sim}_{\mathrm{Co}} H \}$. Note that the reasoning from Remark \ref{rem:coorbit_equiv} applies here as well to guarantee that our definition coincides with \cite{MR4669340}.

 Our arguments in Section \ref{sect:main} repeatedly use the following chain of inclusions, where the first is provided by \cite[Theorem 2.28]{MR4669340}, and the second one follows via Proposition \ref{prop:dual_orbits_coincide}:
\begin{equation} \label{eqn:incl_symmetry}
    N_H \subset \mathcal{S}_H \subset \mathcal{S}_{\mathcal{O}}\,\,.
\end{equation}
\end{remark}

We make the following useful observation regarding the coorbit symmetry group, and this is a special case of \cite[Corollary 4.8]{MR4832531}.
\begin{proposition} \label{prop:coor_conj}
    Let $H_1,H_2< GL(d,\mathbb{R})$ denote two irreducibly admissible matrix groups, and $A \in GL(d,\mathbb{R})$. If $H_1 \sim_{Co} H_2$, then $AH_1 A^{-1} \sim_{Co} A H_2 A^{-1}$.
\end{proposition}

Note that this proposition entails that the conjugation action on subgroups lifts to a conjugation action on the set of coorbit equivalence classes. In this setting, the coorbit symmetry group $\mathcal{S}_H$ is precisely the stabilizer of the coorbit equivalence class of $H$ under conjugation.

\section{Proof of Theorem \ref{thm:main}}

\label{sect:main}

 Up to common open subgroups of finite index, every irreducibly admissible matrix group in dimension two is conjugate to precisely one group from the following list \cite{MR1837417}:
\begin{enumerate}
\item  \textit{Similitude group:}
\[
H_{\smp} = \left\{ \left( \begin{array}{cc} a & b \\ -b & a \end{array} \right) : a^2+b^2 \not= 0 \right\}.
\]
The associated open orbit is given by $\mathcal{O}= \mathbb{R}^{2}\setminus\{0\}$.
\item \textit{Diagonal group:}
\[
H_{\diag} = \left\{ \left( \begin{array}{cc} a & 0 \\ 0 & b \end{array} \right) : ab \not= 0 \right\}.
\] The associated open orbit is given by $\mathcal{O}=\mathbb{R}^{*}\times \mathbb{R}^{*}$.

\item \textit{Shearlet group(s):}
\[
H_{\shr}^c = \left\{ \pm \left( \begin{array}{cc} a & b \\ 0 & a^c  \end{array} \right) : a > 0, b \in \mathbb{R} \right\},  \mbox{ with a unique } c \in \mathbb{R}.
\] The associated orbit is given by 
$\mathcal{O}=\mathbb{R}^{*}\times \mathbb{R}$.
\end{enumerate}

Let us shortly comment about the role of finite index subgroups in this context: Certain irreducibly admissible matrix groups have finite index supergroups which are again irreducibly admissible. Consider, for example, the group,
\[
H^{\textsc 1} = H_{\diag} \cup \left( \begin{array}{cc} 0 & 1 \\ 1 & 0 \end{array} \right) H_{\diag}
\] which properly contains the diagonal group.  
$H^{\textsc 1}$ is clearly distinct from $H_{\diag}$ itself, and has a distinct conjugacy class. Yet one can show, using compactness of the quotient  $H^{\textsc 1}/H_{\diag}$, that $H^{\textsc 1} \sim_{Co} H_{\diag}$, and the analogous observation holds for any conjugate. Hence, the conjugacy class of $H^{\textsc 1}$ does not contribute any new candidates of $\sim_{Co}$-equivalence classes.  

More generally speaking, it has been observed in Remark 4 of \cite{MR1633179} that for every irreducibly admissible matrix group $H$, the connected component $H_0 < H$ of the identity element has finite index in $H$. \cite[Theorem 2.7]{MR1837417} provides that for every irreducibly  admissible matrix group $H$ in dimension two, there exists a conjugate $H'$ of an element from the above list, such that $H$ and $H'$ have the same connected component. This connected component is then of finite index in both groups, which, via \cite[Lemma 4.6]{MR4832531}, yields that $H \sim_{Co} H'$. This observation establishes that the conjugacy classes represented by the above list meet every $\sim_{Co}$-equivalence class, and it justifies focusing our classification on the conjugacy classes represented by the list above.

In the following arguments, we will use the formulation that two irreducibly admissible subgroups $H_1$ and $H_2$ \textit{are conjugate up to finite index} if the respective connected components of the identity are conjugate. 

The next lemma determines the symmetry groups associated to the various representatives.
\begin{lemma} \label{lem:symgroups_repres}
\begin{enumerate}
\item[(a)] If $H= H_{\smp}$, with associated open orbit $\mathcal{O} = \mathbb{R}^2 \setminus \{ 0 \}$, then 
\[
\mathcal{S}_{\mathcal{O}} = \mathcal{S}_H = GL(2,\mathbb{R})
\]
On the other hand
\[
N_H = H \cup  \left( \begin{array}{cc} 1 & 0 \\ 0 & -1 \end{array} \right) H \subsetneq \mathcal{S}_H 
\]
 \item[(b)] If $H = H_{\diag}$, with associated open orbit $\mathcal{O} = \mathbb{R}^* \times \mathbb{R}^*$, then 
 \[ \mathcal{S}_{\mathcal{O}} = H_{\diag} \cup \left( \begin{array}{cc} 0 & 1 \\ 1 & 0 \end{array} \right) H_{\diag} ~.\] As a consequence,
 \[
 N_H = \mathcal{S}_H = \mathcal{S}_{\mathcal{O}}~.
 \]
 \item[(c)] If $H = H_{\shr}^c$ for some suitable $c \in \mathbb{R}$, with associated dual orbit $\mathcal{O} = \mathbb{R}^* \times \mathbb{R}$, then 
 \[
 \mathcal{S}_{\mathcal{O}} = \left\{ \left( \begin{array}{cc} r & s \\ 0 & t \end{array} \right) : rt \not= 0 \right\}~. 
 \]
 As a consequence,
 \[
 N_H = \mathcal{S}_H = \mathcal{S}_{\mathcal{O}}~.
 \]
\end{enumerate}
\end{lemma}
\begin{proof}
The equation $\mathcal{S}_{H} = GL(2,\mathbb{R})$ for $H=H_{\smp}$ was noted in Section 9 of \cite{bib7}.
For parts (b) and (c), it is easy to verify explicitly that $\mathcal{S}_{\mathcal{O}}$ is as described, and that $\mathcal{S}_{\mathcal{O}} \subset N_H$. But then (\ref{eqn:incl_symmetry}) entails the equality $N_H = \mathcal{S}_H = \mathcal{S}_{\mathcal{O}}$.
\end{proof}

\subsection{Proof of part (a)}

If $H$ is any irreducibly admissible matrix group in dimension two, we may assume that $H = AH' A^{-1}$ for some group $H'$ from the above list. If $\mathcal{O},\mathcal{O}'$ denote the associated open dual orbits, a straightforward computation then gives 
\begin{equation}
    \label{eqn:dual_orbit_conjg} H = AH' A^{-1} \Rightarrow \mathcal{O} = A^{-T} \mathcal{O}'~.
\end{equation}
In particular, the number of connected components of the dual orbit is a conjugation invariant. Since $\mathcal{O}'$ has either 1,2, or 4 connected components, the same then holds for $\mathcal{O}$.

\subsection{Proof of part (b)}
\noindent Let $\mathcal{M}$ denote the union of conjugacy classes represented by the list given at the beginning of this section. Hence 
\begin{equation*}
 \displaystyle    \mathcal{M}= [H_{\diag}] \bigcup^\bullet [H_{\smp}] \bigcup^\bullet \left( \bigcup_{c \in \mathbb{R}}^\bullet [H_{\shr}^c] \right),
\end{equation*} where $[H]$ denotes the conjugacy class of $H$. At this point it is useful to recall that the conjugacy equivalence relation is distinct from coorbit equivalence. The crucial first step of our proof is the realization that conjugacy is indeed \textit{coarser}:

\noindent{\textit{\textbf{Step 1}: Assume that $H_1 \sim_{Co} H_2$. Then $H_1$ and $H_2$ are conjugate up to finite index.}

To prove Step 1, first note that the assumption implies that $H_1$ and $H_2$ have the same open dual orbit, by Proposition \ref{prop:dual_orbits_coincide}. If the number of connected components of this open orbit is either 1 or 4, the above classification entails that $H_1$ and $H_2$ are both either conjugate (up to finite index) to $H_{\smp}$ or to $H_{\diag}$. 

In the remaining case, both open orbits have two connected components, and we obtain $H_1 = A_1 H_{\shr}^{c_1} A_1^{-1}$ and $H_2 = A_2 H_{\shr}^{c_2} A_2^{-1}$, for suitable invertible matrices $A_1, A_2$ and $c_1,c_2 \in \mathbb{R}$. Letting $B = A_1^{-1} A_2$, we obtain from Proposition \ref{prop:coor_conj} that 
\[
H_{\shr}^{c_1} \sim_{Co} B H_{\shr}^{c_2} B^{-1}~. 
\] By Proposition \ref{prop:dual_orbits_coincide}, this entails that both groups have the same open dual orbit $\mathcal{O}$, which entails $B \in \mathcal{S}_{\mathcal{O}}$. But then the equality $\mathcal{S}_{\mathcal{O}} = N_{H_{\shr}^{c_2}}$ from Lemma \ref{lem:symgroups_repres} shows $H_{\shr}^{c_1} = H_{\shr}^{c_2}$, and finally $c_1 = c_2$, as desired. Hence, the first step is shown. 

As a consequence of the first step, we just need to determine the coorbit equivalence classes \textit{within} each conjugacy class. 

\noindent{\textit{\textbf{Step 2}: Classifying within the conjugacy class of $H_{\smp}$.}}

For $H = H_{\smp}$, Lemma \ref{lem:symgroups_repres} provides that $\mathcal{S}_H = GL(2,\mathbb{R})$, which expresses that \textit{every} group conjugate to $H_{\smp}$ is already coorbit equivalent to it. Hence $[H]$ is a single coorbit equivalence class. 

\noindent{\textit{\textbf{Step 3}: Classifying within the conjugacy class of $H_{\diag}$.}}

Fix $H = H_{\diag}$, and let $H_1,H_2 \in [H]$. If $H_1 \sim_{Co} H_2$, the equality $\mathcal{S}_H = N_H$ immediately entails that $H_1 = H_2$. Furthermore, the equality $\mathcal{S}_{\mathcal{O}} = N_H$ shows that this is the case precisely if the dual orbits associated to $H_1$ resp. $H_2$ coincides. 

\noindent{\textit{\textbf{Step 4}: Classifying within the conjugacy class of $H_{\shr}^c$.}}

For $H = H_{\shr}^c$, the fact that $\mathcal{S}_H = N_H = \mathcal{S}_{\mathcal{O}}$ again allows to conclude that two members $H_1, H_2$ of the conjugacy class of $H$ are only coorbit equivalent if they are equal, and that this occurs precisely when $H_1$ and $H_2$ have the same open orbit. Hence, step 4 is proved. 

Now steps 1 through 4 settle the proof of part (b): The necessity of $\mathcal{O}_1 = \mathcal{O}_2$ was noted in Proposition \ref{prop:dual_orbits_coincide}. The open dual orbit $\mathcal{O}_1$ has one connected component precisely when the group is conjugate to the similitude group, and any such group is coorbit equivalent to the similitude group (and thus two such groups are coorbit equivalent to each other) by step 2. Likewise, the case of four connected components corresponds to the diagonal group, and is taken care of by step 3, and the case of two connected components by steps 1 and 4. 

\subsection{Proof of part (c):}

The role of connected components and open subgroups of finite index has been clarified in the second paragraph of this section. Given these remarks, part (c) follows from the fact that for $H = H_{\smp}$, one has $\mathcal{S}_H = GL(2,\mathbb{R})$, whereas for the remaining cases, $\mathcal{S}_H = N_H$ entails that conjugates of such a group $H$ are only coorbit equivalent when they are equal. 

\subsection{Proof of part (d):}

The proof of part (b) has already made clear that the conjugacy class of $H_{\smp}$ coincides with its coorbit equivalence class; it is conveniently represented by the similitude group itself. 

In the case of $H = H_{\diag}$, the fact that $\mathcal{S}_H$ is the stabilizer of the conjugation action on the coorbit equivalence classes induces a canonical bijection between these classes and the quotient space $GL(2,\mathbb{R})/\mathcal{S}_H$. 
Hence, the conjugation action of any system $\mathcal{R}$ of representatives modulo $\mathcal{S}_H$ on $H$ will provide the desired system of representatives. 

Such a system is most conveniently constructed using the equality $\mathcal{S}_{\mathcal{O}} = \mathcal{S}_{H} $, which entails that two conjugates of $H_{\diag}$ are coorbit equivalent if and only if their associated dual orbits coincide.  Note that alternatively, we may also use the criterion that the complements of the dual orbits coincide. Observe that any such complement is the union of two distinct lines through the origin, and one can use this description to give an explicit and geometrically intuitive construction of a system of representatives. 

For this purpose we claim that for any union of distinct lines $\mathbb{R} \omega_1 \cup \mathbb{R} \omega_2 \subset \mathbb{R}^2$ through the origin, i.e. with linearly independent vectors $\omega_1$ and $\omega_2$, there exists a unique pair $(\varphi, s) \in [0,\pi) \times [0,\infty)$ such that 
\begin{equation} \label{eqn:mapping_lines}
R_\varphi S_s (\mathbb{R} \times \{0 \} \cup  \{ 0 \} \times \mathbb{R} ) = \mathbb{R} \omega_1 \cup \mathbb{R} \omega_2~.
\end{equation} Here $\mathbb{R}_\varphi$ describes a rotation matrix, and $S_s$ a shearing matrix, given by 
\[
R_\varphi = \left( \begin{array}{cc} \cos(\varphi) & \sin(\varphi) \\ -\sin(\varphi) & \cos(\varphi) \end{array} \right) ~,~ S_s = \left( \begin{array}{cc} 1 & s \\ 0 & 1 \end{array} \right)~. 
\] To prove the existence of these matrices, we let $\vartheta \in (0,\frac{\pi}{2}]$ denote the smaller angle between the two lines $\mathbb{R} \omega_1 \cup \mathbb{R} \omega_2$. Given any $s \ge 0$, the shearing $S_s$ maps $\mathbb{R} \times \{ 0 \}$ onto itself, and $\{ 0 \} \times \mathbb{R}$ onto the line $\mathbb{R} (s,1)^T$, with angle $\vartheta' = \arccos(\frac{s}{\sqrt{1+s^2}}) \in (0,\frac{\pi}{2}]$ between the two image lines. A unique choice of $s \in [0,\infty)$ then guarantees $\vartheta' = \vartheta$. Finally a suitable (and unique) choice of a rotation angle $\varphi \in [0,\pi)$ guarantees the desired equation (\ref{eqn:mapping_lines}). See Figure 1 below.
%
%
%

Hence we have obtained that the family 
\[
A_{\varphi,s} = \left( R_\varphi S_s \right)^{-T} =  \left( \begin{array}{cc} \cos(\varphi) & \sin(\varphi) \\ -\sin(\varphi) & \cos(\varphi) \end{array} \right) \left( \begin{array}{cc} 1 & 0 \\ -s & 1 \end{array} \right)~, (\varphi,s) \in [0,\pi) \times [0,\infty) 
\] defines a system of representatives modulo $\mathcal{S}_{\mathcal{O}} = \mathcal{S}_H$, which then entails that 
\begin{equation} \label{eqn:coorbit_diag}
\left( A_{\varphi,s} H_{\diag} A_{\varphi,s}^{-1} \right)_{(\varphi, s) \in [0,\pi) \times [0,\infty)} 
\end{equation} 
is a system of coorbit equivalence representatives in the conjugacy class of $H_{\diag}$.
A similar but simpler reasoning applies to each of the conjugacy classes of $H_{\shr}^c$, for a fixed parameter $c \in \mathbb{R}$. Here the complement of the dual orbit is a single line through the origin, and two such lines are related by a unique rotation $R_\varphi$, $\varphi \in [0,\pi)$.
Hence the shearlet family supplies the system of coorbit representatives 
\begin{equation}\label{eqn:coorbit_shear}
\left( R_{\varphi} H_{\shr}^c R_{\varphi}^{-1} \right)_{(\varphi,c) \in [0,\pi) \times \mathbb{R}}
\end{equation}

Taking the union over the systems of representatives provides the explicit parameterization in the format postulated in part (d).

\begin{figure}
\begin{center}
\includegraphics[width=0.8\textwidth]{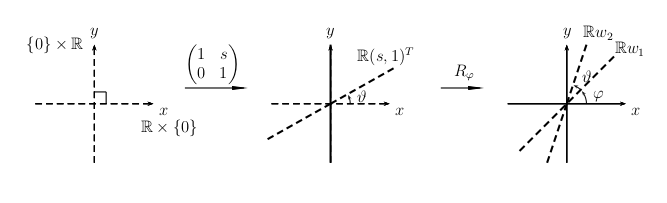}
\end{center}
\caption{Illustration of the cross-section $A_{\varphi,s}$}
\end{figure}

\section*{Appendix: Homogeneous Besov spaces, coorbit spaces and wavelet approximation theory}

\begin{remark}[Besov spaces and image processing] \label{rem:Besov_imag}
 The wavelet characterization of homogeneous Besov spaces is a key property of wavelet systems for applications such as signal and image processing. Briefly, given a function $f$ in a homogeneous Besov space  $\dot{B}_{p,q}^\alpha(\mathbb{R}^d)$, for suitable values of $p,q,\alpha$, it is possible to sharply quantify the decay of the nonlinear $N$-term approximation error of $f$ with respect to a suitably chosen wavelet orthonormal basis. This observation gave rise to rigorous estimates in the analysis of algorithms in applications such as image compression \cite{MR1175690} or denoising \cite{MR1379464}, and established the Besov spaces as a canonical signal class in wavelet signal and image processing. Further background on these aspects can be found in the textbooks \cite{MR2100455, MR1990555, MR2479996, MR1085487}.

 One can argue that current applications of wavelets in signal and image processing, including the ones that are related to machine learning, still rely on the underlying rationale that wavelets succeed in identifying salient signal structures, and that the quantification of wavelet coefficient decay, as used in the characterization of Besov spaces, is a viable rigorous expression of this intuition.
\end{remark}

\begin{remark}[Weighted coorbit spaces] \label{rem:weighted_coorbit}
We provide further construction details regarding coorbit spaces, referring to the sources \cite{bib2, MR1021139,bib7} for the full background. For the specific setting of generalized wavelet coorbit spaces of the type considered in this paper, we also refer to \cite{MR3378833}.

In order to discuss the full scale of homogeneous Besov spaces $\dot{B}_{p,q}^\alpha(\mathbb{R}^d)$, we need to introduce weighted $L^{p,q}$-spaces on a general semidirect product $G = \mathbb{R}^d \rtimes H$, where $H < GL(d,\mathbb{R})$ is irreducibly admissible.
We concentrate on weights of the form
\begin{equation} \label{eqn:def_vs}
 v_s(x,h) = |{\rm det}(h)|^{-s}~,
\end{equation} and define, for $1 \le p, q < \infty$ the associated weighted $L^{p,q}$ norm of $F: G \to \mathbb{C}$ as
\[
 \| F \|_{L^{p,q}_{v_s}} = \int_{H} \left( \int_{\mathbb{R}^d} |{\rm det}(h)|^{-sq} |F(x,h)|^q dx \right)^{p/q}~.
\] This defines the Banach space $L^{p,q}_{v_s}(G)$ of weighted $L^{p,q}$-functions. We write $L^{p,q}(G) := L^{p,q}_{v_0}(G)$, for the constant weight $v_0\equiv 1$. Note that $L^{p,p}(G) = L^p(G)$.

These weighted $L^{p,q}$-spaces fulfill the prerequisites for coorbit space theory, and thus allow to define the associated coorbit spaces $Co_H(L^{p,q}_{v_s}(G))$, for general irreducibly admissible dilation groups \cite{MR3378833}. In other words, the coorbit space norm of a signal becomes
\[
 \| f \|_{Co_H(L^{p,q}_{v_s})} =  \int_{H} \left( \int_{\mathbb{R}^d} |{\rm det}(h)|^{-sq} |W_\psi f (x,h)|^q dx \right)^{p/q}~,
\]
for suitable choices of $\psi$.

In order to rigorously define coorbit spaces, one introduces a suitable space $\mathcal{H}^1_w \subset L^2(G)$ of \textit{analyzing wavelets} \cite[4.1]{bib2}, using a suitably chosen control weight $w$ on $G$. The {\em reservoir space} $\mathcal{H}^{1,\sim}_w$ is then defined as the conjugate dual space of $\mathcal{H}^1_w$. One then obtains a so-called Banach-Gelfand triple, i.e. an embedding of Banach spaces $\mathcal{H}^{1}_w \hookrightarrow L^2(\mathbb{R}^d) \hookrightarrow \mathcal{H}^{1,\sim}_w$, where the second embedding is obtained by dualization of the first one. Furthermore, the wavelet transform $W_\psi f : G \to \mathbb{C}$ can now be meaningfully defined for $(\psi, f) \in \mathcal{H}^{1}_w \times \mathcal{H}^{1,\sim}_w$. Now the coorbit space $Co_H(L^{p,q}_{v_s})$ is rigorously defined by picking any nonzero $\psi \in \mathcal{H}^1_w$ and letting
\[
 Co_H(L^{p,q}_{v_s}) = \left\{ f \in \mathcal{H}^{1,\sim}_w : W_\psi f \in L^{p,q}_{v_s}(G) \right\} ~.
\]
\end{remark}

\begin{remark}[Homogeneous Besov spaces as coorbit spaces] \label{rem:Besov_coorbit}
Coorbit space theory was developed as a group-theoretic unified approach to Besov and modulation spaces. Hence it was noticed early on that homogeneous Besov spaces are indeed coorbit spaces associated to the similitude group $H = \mathbb{R}^+ \cdot SO(d)$, more precisely
\begin{equation} \label{eqn:besov_coorbit}
 \dot{B}_{p,q}^\alpha(\mathbb{R}^d) = Co_{H}(L^{p,q}_{v_s})~,~ s = \alpha + \frac{1}{2}-\frac{1}{q}~.
\end{equation}
We refer to the foundational source \cite{bib2} for the statement in the case $p=q$ and $d=1$, and to \cite[Theorem 4.5]{MR2898467} for general $d$. (Note that \cite{MR2898467} uses $H = \mathbb{R}^+$ in combination with isotropic wavelets which leads to the same coorbit spaces, and a slightly different parametrization of the weight function.)

It is important to note that the equality (\ref{eqn:besov_coorbit}) relies on a suitable \emph{identification} of objects in different reservoir spaces. Typically, Besov spaces $\dot{B}_{p,q}^\alpha(\mathbb{R}^d)$ are defined as subspaces of the space of tempered distributions (modulo polynomials), whereas we saw in the previous remark that coorbit spaces $Co_H(L^{p,q}_{v_s})$ are constructed as subspaces of a conjugate dual space, which depends on the group $G = \mathbb{R}^d \rtimes H$. We refer to \cite{MR2898467} for a description of how the identification is carried out in the proof of (\ref{eqn:besov_coorbit}).

An important part of coorbit theory is the \textit{discretization theory} \cite{MR1021139}, providing discrete Banach frames and associated characterization results for coorbit spaces based on discrete frame expansions. As a consequence, the derivation of nonlinear wavelet approximation rates of elements of Besov spaces, as sketched in Remark \ref{rem:Besov_imag}, can be transferred to elements $Co_H(L^{p,q}_{v_s})$, with discrete wavelet frame expansions associated to $H$ instead of the similitude group. See e.g. \cite[Section 4.4]{MR2543193} for a discussion in the case of shearlet coorbit spaces, and \cite[Section 1]{MR3452925} for the general case.

In summary, the coorbit spaces $Co_H(L^{p,q}_{v_s})$ provide a description of the approximation theoretic properties of wavelet systems associated to $H$, in the same way that homogeneous Besov spaces achieve this for wavelet systems associated to the similitude group.
\end{remark}

We finish with a result justifying our Definition \ref{defn:coorbit_equiv} of coorbit equivalence, which focusses on $Co_H(L^p)$ as reference spaces. Most of the theorem is already known from \cite{MR4464548}; however, item $(f)$ extends the statement to certain weighted mixed $L^p$ norms, and is new. For the terminology used in item (d) of the theorem, as well as in the proof, we refer to \cite{MR4464548}.

\begin{theorem} \label{thm:coorbit_equiv_dual_orbits}
 Let $H_1, H_2 < GL(\mathbb{R}^d)$ denote irreducibly admissible matrix groups, and let $\mathcal{O}_1,\mathcal{O}_2$ denote the associated open dual orbits. Then the following are equivalent:
 \begin{enumerate}
  \item[(a)] For all $0 < p, q \le \infty$ and for all $f \in L^2(\mathbb{R}^d)$,
  	\begin{equation*}
		\|f\|_{Co_{H_1}({L}^{p,q})} \asymp \|f\|_{Co_{H_2}({L}^{p,q})}~,
	\end{equation*}
	in the generalized sense that one side is infinity if and only if the other side is also.
  \item[(b)] For all $1 \le p,q \le 2$: $Co(L^{p,q}(\mathbb{R}^d \rtimes H_1)) = Co(L^{p,q}(\mathbb{R}^d \rtimes H_2))$, as subspaces of $L^2(\mathbb{R}^d)$.
  \item[(c)] There exists $1 \le p,q \le 2$ with $(p,q) \not= (2,2)$, such that $Co(L^{p,q}(\mathbb{R}^d \rtimes H_1)) = Co(L^{p,q}(\mathbb{R}^d \rtimes H_2))$, as subspaces of $L^2(\mathbb{R}^d)$.
  \item[(d)] $\mathcal{O}_1 =  \mathcal{O}_2$, and the coverings induced by $H_1$ and $H_2$ on the common open orbit are weakly equivalent.
  \item[(e)] $H_1$ and $H_2$ are coorbit equivalent.
  \item[(f)] For all $0 < p,q \le \infty$, all $s \in \mathbb{R}$ and all $f \in L^2(\mathbb{R}^d)$,
  	\begin{equation*}
		\|f\|_{Co_{H_1}({L}^{p,q}_{v_s})} \asymp \|f\|_{Co_{H_2}({L}^{p,q}_{v_s})}~,
	\end{equation*} with $v_s$ defined in (\ref{eqn:def_vs}), 
	in the generalized sense that one side is infinity if and only if the other side is also.
 \end{enumerate}
\end{theorem}
\begin{proof}
 The equivalence of $(a)$ through $(d)$ is \cite[Theorem 2.18]{MR4464548}; note the slight difference between our Definition \ref{defn:coorbit_equiv} and \cite[Definition 2.16]{MR4464548}. The implications $(a) \Rightarrow (e) \Rightarrow (c)$ as well as  $(f) \Rightarrow (a)$ are immediate. Hence it remains to prove $(d) \Rightarrow (f)$.

  The following is essentially an adaptation of the proof of the direction $(d) \Rightarrow (a)$ to the weighted setting. We sketch the argument, using results and terminology from \cite{MR4464548}.

  Let $\mathcal{O} = \mathcal{O}_1 = \mathcal{O}_2$, which is an open subset of $\mathbb{R}^d$ with complement of measure zero. Now the Banach-Gelfand triple $\mathcal{H}^{1}_w \hookrightarrow L^2(\mathbb{R}^d) \hookrightarrow \mathcal{H}^{1,\sim}_w$ allows to define an operator
  \[
   \mathcal{F}_1:  \mathcal{H}^{1,\sim}_w \to \mathcal{D}'(\mathcal{O})~,
  \] see \cite[Corollary 10]{bib7}. Here $\mathcal{D}'(\mathcal{O})$ denotes the set of linear functionals on $C_c^\infty (\mathcal{O})$, defined and topologized as usual \cite[6.7]{MR1157815}. The operator coincides with the Plancherel transform on $L^2(\mathbb{R}^{d})$ \cite[Remark 2.15]{MR4464548}. Furthermore, $\mathcal{F}_1$, when restricted to $Co_H(L^{p,q}_{v_s})$, will give rise to an \textit{isomorphism} with a suitable \textit{decomposition space}, which we describe next.

  The dual action of $H_1$ gives rise to a so-called \textit{induced covering} of $\mathcal{O}$,
 \[
  \mathcal{P} = (P_i)_{i \in I}~, P_i = h_i^{-T} U_1~,
 \] with $U_1 \subset \mathcal{O}$ a suitable relatively compact, connected open set, and $(h_i)_{i \in I} \subset H_1$  a \textit{well-spread} family; see \cite[Subsection 2.2]{MR4464548}.

 Fixing any $s \in \mathbb{R}$, we define a weight on the index set $I$ via $u_s^1 : I \to \mathbb{R}^+$,
 \begin{equation} \label{eqn:defn_u1}
 u_s^1(i) = |{\rm det}(h_i)|^{\frac{1}{2}-\frac{1}{q}} v_s(h_i) = |{\rm det}(h_i)|^{\frac{1}{2}-\frac{1}{q}-1}
  = \frac{|P_i|^{1-\frac{1}{2}+\frac{1}{q}}}{|U_1|^{1-\frac{1}{2}+\frac{1}{q}}}~,
 \end{equation} with $|A|$ denoting Lebesgue measure of the set $A \subset \mathbb{R}^d$.

Then, by \cite[Theorem 2.13]{MR4464548}, the restriction
 \[
  \mathcal{F}_1 : Co(L^{p,q}_{v_s}(\mathbb{R}^d \rtimes H_1)) \to \mathcal{D}(\mathcal{P},L^p,\ell^q_{u^1})~,
 \] is a topological isomorphism of Banach spaces, where the right-hand side of the definition is a so-called \textit{Fourier-side decomposition space}  \cite[Subsection 2.1]{MR4464548}, a subspace of $\mathcal{D}'(\mathcal{O})$.

The same reasoning can also be applied to $H_2$: The dual action of $H_2$ induces a covering $\mathcal{Q} = (Q_j)_{j \in J}$ of $\mathcal{O}$, based on a well-spread family $(g_j)_{j \in J} \subset H_2$ via $Q_j = g_j^{-T} U_2$. We define an associated weight $u_s^1 : I \to \mathbb{R}^+$,
 \begin{equation} \label{eqn:defn_u2}
 u_s^2(j) = |{\rm det}(g_j)|^{\frac{1}{2}-\frac{1}{q}} v_s(g_j) = |{\rm det}(g_j)|^{\frac{1}{2}-\frac{1}{q}-1}
  = \frac{|Q_j|^{1-\frac{1}{2}+\frac{1}{q}}}{|U_2|^{1-\frac{1}{2}+\frac{1}{q}}}~,
 \end{equation}
 and thus obtain a topological isomorphism
 \[
   \mathcal{F}_2 : Co(L^{p,q}_{v_s}(\mathbb{R}^d \rtimes H_2)) \to \mathcal{D}(\mathcal{Q},L^p,\ell^q_{u^2})~,
 \] as restriction of a suitable operator
\[  \mathcal{F}_2:  \mathcal{H}^{1,\sim}_w \to \mathcal{D}'(\mathcal{O})~.
\]
Note that the space $\mathcal{H}^{1,\sim}_w$ is not the same as the identically denoted space associated to $H_1$. However, $\mathcal{F}_2$ again restricts to the Plancherel transform on $L^2(\mathbb{R}^d)$.

 We now use the assumption that the induced coverings are weakly equivalent; see e.g. \cite[Definition 2.2]{MR4464548} for a definition. This property entails, via \cite[Lemma 2.8]{MR4464548}, that
 \[
  \forall i \in I, j \in J ~:~ P_i \cap Q_j \not= \emptyset \Rightarrow C_1^{-1} |P_i| \le |Q_j| \le C_1 |P_i|~,
 \] with a constant $C_1 \ge 1$. Combining this fact with  (\ref{eqn:defn_u1}) and (\ref{eqn:defn_u2}) yields that
\[
  \forall i \in I, j \in J ~:~ P_i \cap Q_j \not= \emptyset \Rightarrow C_2^{-1} u^1(i) \le u^2(j) \le C_2 u^1(i)~,
 \] for a constant $C_2>1$.
 This property, together with weak equivalence of the coverings, finally ensures via \cite[Lemma 2.4]{MR4464548} that
 \[
  \mathcal{D}(\mathcal{P},L^p,\ell^q_{u^1}) = \mathcal{D}(\mathcal{Q},L^p,\ell^q_{u^2})~.
 \] This equality, the topological isomorphisms $ \mathcal{F}_1 : Co(L^{p,q}_{v_s}(\mathbb{R}^d \rtimes H_1)) \to \mathcal{D}(\mathcal{P},L^p,\ell^q_{u^1})$ and $ \mathcal{F}_2 : Co(L^{p,q}_{v_s}(\mathbb{R}^d \rtimes H_2)) \to \mathcal{D}(\mathcal{P},L^p,\ell^q_{u^2})$, and the fact that $\mathcal{F}_1$ and $\mathcal{F}_2$ coincide on $L^2(\mathbb{R}^d)$, finally entail the norm equivalence in $(f)$.
\end{proof}

\section*{Concluding Remarks}
Wavelet transforms and their relatives have been repeatedly proposed as a tool in image processing and analysis (see e.g. \cite{MR2100455}), and our focus on the two-dimensional case can in part be understood as a contribution towards a better understanding of the relative merits of various wavelet constructions for image processing purposes. Moreover, this problem provides a convenient showcase for the demonstration of recently developed methods from coorbit theory, in particular for the purpose of classification. We point out that the approach presented here can be adapted to higher dimensions. An analogous study of the three-dimensional case is currently in preparation, which poses significantly greater challenges than dimension two, mostly due to the much larger reservoir of conjugacy classes in dimension three. 
\section{Acknowledgments}
We thank the reviewers for useful comments and suggestions.
Vaishakh Jayaprakash is supported
by the INSPIRE Ph.D. Fellowship of the Department of Science and Technology, Government of India. Noufal Asharaf is supported by the project RUSA 2.0, T3A from the Cochin University of Science and Technology.
\bibliographystyle{abbrv}
\bibliography{sn-bibliography}



\end{document}